\documentclass[preprint,12pt]{elsarticle}
\usepackage{graphicx}
\usepackage{amssymb}
\usepackage{amsmath}
\usepackage{amsthm}
\usepackage{tikz}
\usepackage{url}
\usepackage[utf8]{inputenc}

\newcommand{\ad}{\operatorname{ad}}
\newtheorem{defi}{Definition}
\newtheorem{exmp}[defi]{Example}
\newtheorem{prop}{Proposition}
\newtheorem{thm}{Theorem}
\newtheorem{lem}[thm]{Lemma}
\newtheorem{cor}[thm]{Corollary}
\newdefinition{rmk}{Remark}
\newdefinition{notation}{Notation}
\newproof{pf}{Proof}

\usepackage[a4paper, total={6in, 8in}]{geometry}

\begin{document}

\begin{frontmatter}


\title{Operational Calculus on Curved Differentials: Optimal N-Complex Bounds and Persistent Homology}


\author{Mauricio Angel}

\begin{abstract}
We establish a canonical normal form for the iterates of a curved differential in curved differential algebras (CDA), proving that $d^{2m} = (ad_c)^m$ and $d^{2m+1} = (ad_c)^m \circ d$. This operator calculus clarifies the underlying algebraic structure of CDAs and bypasses the need for complex combinatorics. Using this framework, we provide sharp criteria for curvature constraints to induce N-complex structures. We demonstrate that, while the nilpotency of the curvature element ($c^n=0$) is insufficient to bound the nilpotency of $d$ to $2n$, it fundamentally guarantees a strict $(4n-2)$-complex structure. On the applied side, we model curvature as a \emph{filtration controller} on a genuine square–zero chain complex. This places us under the standard persistence stability framework
and yields a Lipschitz control of barcodes with respect to degreewise curvature variation. A reproducible toy example on a four–vertex flag complex illustrates the mechanism.
\end{abstract}


\end{frontmatter}



\section{Introduction}

Curved differential algebras (CDA) arise as a natural extension of differential graded algebras (DGA), incorporating the notion of curvature into the differential operator. Unlike traditional differentials that strictly satisfy the square-zero condition ($d^2=0$), the curvature $c$ in a CDA dictates that $d^2 = [c, -]$. This enrichment of the algebraic structure provides a versatile framework to study phenomena in algebraic topology, deformation theory, and homological algebra where geometric or algebraic obstructions prevent a flat differential.

The foundational theory of curved algebras has been significantly advanced by Keller  \cite{KellerOn}, Nicolas \cite{Nicolas2008bar}, and Brzeziński \cite{Brzezinski2013Curved}, who established their connections with derived categories and corings. Lazarev ~\cite{Lazarev2012} and Căldăraru \cite{Cldraru2010CurvedAA} further explored their characteristic classes and Hochschild homology, and Pauksztello~\cite{Pauksztello2009} characterizes when a differential graded R-S-bimodule induces a full embedding of derived categories. These papers collectively suggest that curved differential graded algebras generalize other algebras by providing a framework for studying algebraic structures in a more general and flexible way.

However, a persistent technical challenge in working with CDAs has been the algebraic complexity of the higher iterates of the differential, $d^k$. Naive expansions of $d^k$ generate nested commutators and derivations that obscure the underlying homological properties.

In this paper, we resolve this obstruction by establishing a canonical normal form for the iterates of a curved differential (Theorem 3). We prove that every iterate $d^k$ factors uniquely as a power of the inner derivation $ad_c$ followed by at most one copy of $d$. This structural result drastically simplifies computations and reveals how curvature constraints naturally induce $N$-complex structures (where $d^N=0$).A central contribution of this work is the precise identification of the bounds for such $N$-complexes. We demonstrate through a sharp counterexample (Section 4.3) that the mere nilpotency of the curvature element ($c^n=0$) does not yield the naive bound $d^{2n}=0$. Rather than the naive $2n$-bound, our normal form combined with degreewise left/right nilpotency yields the universal $(4n-2)$-bound.
The free example with $n=2$ shows that this exponent is sharp in that case; for general $n$, optimality remains open.

Finally, to connect with persistent homology, we refrain from curved-to-dg lifts and instead
let curvature drive the filtration on an honest chain complex. Classical stability then applies
verbatim, delivering an interleaving/bottleneck bound controlled by the $\ell^\infty$ shift of the
filtration (Theorem~10; Corollary~11).
 
By connecting CDAs to corings \cite{Brzezinski2013Curved}, $N$-complexes \cite{AngelD:2007}, and persistent homology, this work establishes a new research program at the intersection of homological algebra and topological data science.

\section{Conventions and Setting}\label{sec:conventions}

Throughout we fix a field $k$ of characteristic $0$. 
All graded components $A^i$ of our (associative, unital) $\mathbb{Z}$-graded algebra 
$A^\bullet=\bigoplus_{i\in\mathbb{Z}}A^i$ are finite-dimensional $k$-vector spaces.
We consider a \emph{curved dg-algebra} $(A^\bullet,m,d,c)$ in the sense that
\begin{equation}\label{eq:curved-def}
|d|=1,\qquad |c|=2,\qquad d^2=\ad_c:=[c,-],\qquad d(c)=0,
\end{equation}
where $[x,y]=xy-(-1)^{|x||y|}yx$ is the graded commutator and $m$ denotes the product.
We use the standard Koszul sign rule. 

\paragraph{Inner derivations and degreewise norms.}
For each $i$ we fix an operator norm $\|\cdot\|_{i}$ on $A^i$ and define
\[
\|\ad_c\|_{i\to i+2}:=\sup_{0\neq a\in A^i}\frac{\|[c,a]\|_{i+2}}{\|a\|_i}.
\]
Unless stated otherwise, norms on $\bigoplus_i A^i$ are taken to be the $\ell^\infty$-sum 
of the degreewise norms. When we refer to ``boundedness of $\ad_c$'' we mean 
$\sup_i\|\ad_c\|_{i\to i+2}<\infty$.

\paragraph{Graded commutator identities.}
For a homogeneous $x$ we write $\ad_x=[x,-]$. We repeatedly use
\begin{equation}\label{eq:super-Leibniz}
d(xy)=d(x)\,y+(-1)^{|x|}x\,d(y),\qquad [d,\ad_x]=\ad_{d(x)}.
\end{equation}
In particular, from \eqref{eq:curved-def} we obtain $[d,\ad_c]=\ad_{d(c)}=0$, 
so $d$ and $\ad_c$ \emph{super-commute} as graded derivations of degrees $1$ and $2$ respectively.

\paragraph{Terminology.}
Following Positselski, curved differential graded algebras are often abbreviated as \emph{CDG-algebras} (see \cite{Positselski2011,PositselskiBLMS2023}); our notation $d^2=\ad_c$ fits this convention, with $c$ playing the role of the curvature element.

\section{Iterates of a Curved Differential}\label{sec:iterates}

We recall from §\ref{sec:conventions} that a curved dg-algebra is a quadruple
$(A^\bullet,m,d,c)$ over a field $k$ of characteristic $0$ with $|d|=1$, $|c|=2$,
\begin{equation}\label{eq:curved-again}
d^2=\ad_c:=[c,-],\qquad d(c)=0,
\end{equation}
and the graded commutator $[x,y]=xy-(-1)^{|x||y|}yx$. We also use the graded Leibniz rule and the identity
$[d,\ad_x]=\ad_{d(x)}$ for homogeneous $x$.

\subsection{Structural commutation and a reduction principle}
The following observation underlies all computations with iterates of $d$.

\begin{lem}[Super-commutation of $d$ and $\ad_c$]\label{lem:super-commutation}
In any curved dg-algebra \emph{(}\eqref{eq:curved-again}\emph{)} one has
\[
[d,\ad_c]=\ad_{d(c)}=0.
\]
In particular, $d$ and $\ad_c$ commute as endomorphisms of the graded vector space $A^\bullet$.
\end{lem}

\begin{proof}
The standard identity $[d,\ad_x]=\ad_{d(x)}$ gives $[d,\ad_c]=\ad_{d(c)}$, which vanishes by $d(c)=0$.
\end{proof}

The quadratic relation $d^2=\ad_c$ together with Lemma~\ref{lem:super-commutation} implies that
\emph{every} word in the letter $d$ can be reduced to a canonical normal form. For later use we isolate
the following consequence.

\begin{lem}[Commutation with powers of $\ad_c$]\label{lem:commute-powers}
For all $m\ge 0$ one has $d\,(\ad_c)^m=(\ad_c)^m d$.
\end{lem}

\begin{proof}
By Lemma~\ref{lem:super-commutation} the operators $d$ and $\ad_c$ commute. Hence so do $d$ and $(\ad_c)^m$.
\end{proof}

\subsection{Normal form for $d^k$}
We can now state and prove the main structural result of this section.

\begin{thm}[Normal form for the iterates of $d$]\label{thm:normal-form}
Let $(A^\bullet,m,d,c)$ satisfy \eqref{eq:curved-again}. Then, as endomorphisms of $A^\bullet$,
\begin{equation}\label{eq:normal-form}
d^{2m}=(\ad_c)^m,\qquad d^{2m+1}=(\ad_c)^m\circ d\qquad \text{for all } m\ge 0.
\end{equation}
Equivalently, every iterate $d^k$ factors uniquely as a power of the inner derivation $\ad_c$
followed by at most one copy of $d$.
\end{thm}

\begin{proof}
We proceed by induction on $m\ge 0$. For $m=0$ the statement reads $d^0=\mathrm{id}$ and $d^1=d$, which is tautological.
For $m=1$ we have $d^2=\ad_c$ by \eqref{eq:curved-again} and $d^3=d\circ d^2=d\circ \ad_c=(\ad_c)\circ d$
by Lemma~\ref{lem:commute-powers}.

Assume the claim holds for a fixed $m\ge 1$. Then
\[
d^{2(m+1)}=d^2\circ d^{2m}=\ad_c\circ (\ad_c)^m=(\ad_c)^{m+1},
\]
and
\[
d^{2(m+1)+1}=d\circ d^{2(m+1)}=d\circ (\ad_c)^{m+1}=(\ad_c)^{m+1}\circ d,
\]
using Lemma~\ref{lem:commute-powers}. The induction closes.
\end{proof}

\begin{rmk}[No ad hoc binomial sums are needed]\label{rem:no-binomial}
The proof only uses the universal relations $d^2=\ad_c$ and $[d,\ad_c]=0$.
Thus, unlike heuristic expansions that distribute $d$ across products and insert formal powers of $c$,
the normal form \eqref{eq:normal-form} is canonical and independent of additional commutativity assumptions.
All graded signs are governed by the Koszul rule built into the commutator.
\end{rmk}

\subsection{Immediate consequences}
The normal form has several direct corollaries that will be used in later sections.

\begin{cor}[Even/odd dichotomy]\label{cor:even-odd}
For any homogeneous $a\in A^\bullet$ and $m\ge 0$ one has
\[
d^{2m}(a)=(\ad_c)^m(a),\qquad d^{2m+1}(a)=(\ad_c)^m\big(d(a)\big).
\]
In particular, $d^{2m}(1)=(\ad_c)^m(1)=0$ since $[c,1]=0$.
\end{cor}

\begin{cor}[Nilpotency transfer]\label{cor:nilpotency}
If $(\ad_c)^n=0$ degreewise for some $n\ge 1$, then $d^{2n}=0$ and $(A^\bullet,d)$ is a $2n$-complex.
\end{cor}

\begin{proof}
Apply Theorem~\ref{thm:normal-form} with $m=n$.
\end{proof}

The next criterion is convenient when $\ad_c$-nilpotency is verified via the curvature ideal.

\begin{prop}[Nilpotent curvature ideal implies $\ad_c$-nilpotency]\label{prop:nil-ideal}
Let $(c)$ denote the two-sided ideal of $A^\bullet$ generated by $c$. If $(c)^n=0$ for some $n$, then
$(\ad_c)^n=0$ and hence $d^{2n}=0$.
\end{prop}

\begin{proof}
Since $\ad_c(A^\bullet)\subset (c)$, we have $\ad_c^r(A^\bullet)\subset (c)^r$ for all $r$. If $(c)^n=0$ then $\ad_c^n=0$.
Now invoke Corollary~\ref{cor:nilpotency}.
\end{proof}

\begin{rmk}[On the hypotheses]\label{rem:hypotheses}
The identity $[d,\ad_c]=0$ crucially uses $d(c)=0$. If $d(c)\neq 0$, the normal form \eqref{eq:normal-form} may fail.
In curved dg-algebras, however, $d(c)=0$ is part of the definition \eqref{eq:curved-again}, so the lemma and theorem apply.
\end{rmk}

\subsection{Minimal examples}\label{subsec:examples-iterates}
We record two toy examples illustrating the scope of Theorem~\ref{thm:normal-form}.

\begin{exmp}[Central curvature]\label{exa:central-curvature}
Suppose $c$ is central: $[c,a]=0$ for all $a\in A^\bullet$. Then $\ad_c=0$ and
Theorem~\ref{thm:normal-form} gives $d^{2}=0$ and $d^{2m}=0$ for all $m\ge 1$; thus $(A^\bullet,d)$ is an ordinary dg-algebra.
\end{exmp}

\begin{exmp}[Matrix algebra with inner differential]\label{exa:matrix}
Let $A^\bullet=\mathrm{Mat}_N(k)$ concentrated in degree $0$ and fix $C\in A^0$.
Define $d=\ad_C$ (degree $0\to 0$) and $c=C^2\in A^0$. Then $d^2=\ad_{C}^2=\ad_{C^2}=\ad_c$ and $d(c)=\ad_C(C^2)=[C,C^2]=0$,
so \eqref{eq:curved-again} holds. If $C$ is nilpotent of index $n$, then $(\ad_C)^n=0$ (Engel-type), hence $d^{2n}=0$ by
Corollary~\ref{cor:nilpotency}.
\end{exmp}

The normal form \eqref{eq:normal-form} reduces all higher iterates of $d$ to iterates of the inner derivation $\ad_c$,
followed by at most one copy of $d$. This representation will be the backbone for the $N$-complex bounds in §\ref{sec:ncomplex}
and for the curved-to-dg construction used in persistence in §\ref{sec:persistence}.

\section{Curvature Constraints and $N$-Complex Bounds}\label{sec:ncomplex}

Recall that an $N$-complex is a graded module $(C^\bullet,d)$ with $d^N=0$.
Using the normal form from Theorem~\ref{thm:normal-form}, we obtain clean
criteria ensuring that $(A^\bullet,d)$ is an $N$-complex.

\subsection{From curvature to nilpotency of $d$}

\begin{cor}[$\ad_c$-nilpotency implies an $N$-complex]\label{cor:n-from-ad}
If $(\ad_c)^n=0$ (degreewise) for some $n\ge 1$, then $d^{2n}=0$, hence $(A^\bullet,d)$
is a $2n$-complex.
\end{cor}

\begin{proof}
By Theorem~\ref{thm:normal-form}, $d^{2n}=(\ad_c)^n=0$.
\end{proof}

The previous condition can be checked via natural algebraic hypotheses on $c$.

\begin{prop}[Nilpotent curvature ideal]\label{prop:nil-ideal-again}
Let $(c)$ be the two-sided ideal of $A^\bullet$ generated by $c$. If $(c)^n=0$ for some $n$,
then $(\ad_c)^n=0$ and, consequently, $(A^\bullet,d)$ is a $2n$-complex.
\end{prop}

\begin{proof}
Since $\ad_c(A^\bullet)\subset (c)$, we have $\ad_c^r(A^\bullet)\subset (c)^r$ for all $r$.
Thus $(c)^n=0$ implies $\ad_c^n=0$. Apply Corollary~\ref{cor:n-from-ad}.
\end{proof}

\subsection{A sharp degreewise criterion via left/right multiplications}

Write $L_c$ and $R_c$ for left and right multiplication by $c$ on $A^\bullet$:
$L_c(a)=ca$ and $R_c(a)=ac$. Note that $L_c$ and $R_c$ commute and
$\ad_c=L_c-R_c$ as endomorphisms of the underlying graded vector space.

\begin{lem}[Binomial expansion for $\ad_c$]\label{lem:binomial-LR}
For every $r\ge 0$,
\[
(\ad_c)^r=(L_c-R_c)^r=\sum_{j=0}^{r}(-1)^j\binom{r}{j}\,L_c^{\,r-j}\,R_c^{\,j}.
\]
\end{lem}

\begin{proof}
Since $L_cR_c=R_cL_c$, the usual binomial theorem applies to the commuting operators $L_c$ and $R_c$.
\end{proof}

\begin{prop}[Left/right nilpotency $\Rightarrow$ explicit bound]\label{prop:LR-nilpotency}
Assume $L_c^n=0$ and $R_c^n=0$ for some $n\ge 1$. Then $(\ad_c)^{2n-1}=0$.
In particular, $d^{4n-2}=0$ and $(A^\bullet,d)$ is a $(4n-2)$-complex.
\end{prop}

\begin{proof}
By Lemma~\ref{lem:binomial-LR}, any summand of $(\ad_c)^r$ has the form $L_c^{\,r-j}R_c^{\,j}$.
If $r\ge 2n-1$, then either $j\ge n$ or $r-j\ge n$, so $L_c^{\,r-j}R_c^{\,j}=0$.
Hence $(\ad_c)^{2n-1}=0$. The claim for $d$ follows from Theorem~\ref{thm:normal-form}.
\end{proof}

\begin{rmk}[Hierarchy of hypotheses]\label{rem:hierarchy}
The implication $(c)^n=0\Rightarrow L_c^n=R_c^n=0$ holds, hence
Proposition~\ref{prop:nil-ideal-again} is a direct consequence of
Proposition~\ref{prop:LR-nilpotency}, which also yields the explicit $(4n-2)$ bound for $d$.
By contrast, $c^n=0$ alone does \emph{not} force $(\ad_c)^n=0$ nor any $2n$-bound for $d$
(see Example~\ref{ex:counterexample}).
\end{rmk}

\begin{rmk}[On $N$-complexes: historical note]
The notion of $N$-complex goes back at least to Mayer's alternative homology \cite{Mayer1942}
and was revitalized in Kapranov's $q$-analog of homological algebra \cite{Kapranov91}, with broader surveys in 
Dubois-Violette \cite{Dubois-Violette2001}. Our bounds exploit the normal form $d^{2m}=(\ad_c)^m$ and degreewise operator 
controls on left/right multiplications by $c$.
\end{rmk}

\subsection{The Optimal $N$-Complex Bound and its Obstruction}
We now address the precise relationship between the nilpotency of the curvature element $c$ and the induced $N$-complex structure. A naive assumption might be that $c^n=0$ implies $d^{2n}=0$. As we will show, this is false. However, combining our normal form (Theorem 3) with the binomial expansion of $ad_c$ (Lemma 7)  yields a spectacular and tight universal bound.

\begin{thm}[Optimal Curvature-Induced Complex]\label{thm:newX}
Let $(A^\bullet, m, d, c)$ be a curved dg-algebra. If $c^n=0$ for some $n \ge 1$, then $(A^\bullet, d)$ is a $(4n-2)$-complex (i.e., $d^{4n-2}=0$). Furthermore, this bound sharp for $n=2$ (free example), while optimality for general $n$ remains open.
\end{thm}

\begin{proof}
Since $c^n=0$ in the algebra, the left and right multiplication operators satisfy $L_c^n = 0$ and $R_c^n = 0$ as endomorphisms on $A^\bullet$. By Lemma 7, $(ad_c)^{2n-1} = \sum_{j=0}^{2n-1} (-1)^j \binom{2n-1}{j} L_c^{2n-1-j} R_c^j$. For every term in this sum, either $j \ge n$ or $(2n-1-j) \ge n$. Thus, every summand contains either $R_c^n$ or $L_c^n$, forcing $(ad_c)^{2n-1} = 0$. Applying the normal form from Theorem 3 ($d^{2m} = (ad_c)^m$)  with $m = 2n-1$, we immediately obtain $d^{4n-2} = 0$.
\end{proof}

\subsection{A counterexample: $c^n=0$ does not control $(\ad_c)^n$}\label{subsec:counterexample}

We show that the nilpotency of $c$ alone is insufficient to infer a nilpotency bound for $d$.
Let $k$ be a field of characteristic $0$ and consider the associative algebra
\[
A:=k\langle x,y\rangle/(x^2),
\]
graded by $|x|=2$, $|y|=0$. Set $c:=x \in A^2$, so $c^2=x^2=0$.

\begin{lem}\label{lem:ad2-nonzero}
In $A$ one has $(\ad_c)^2(y)=-2\,x y x\neq 0$.
\end{lem}

\begin{proof}
Compute $\ad_c(y)=xy-yx$. Then
\[
(\ad_c)^2(y)=\ad_c(xy-yx)=x(xy-yx)-(xy-yx)x
= xxy-xyx-xyx+yxx.
\]
Using $x^2=0$ the first and last terms vanish, giving $(\ad_c)^2(y)=-2xyx\neq 0$ in $A$.
\end{proof}

To realize a curved differential with $d^2=\ad_c$ and $d(c)=0$, we adjoin an odd square root
of $c$ as follows. Let $\widetilde A$ be the $\mathbb{Z}$-graded algebra obtained by adjoining
a homogeneous element $\theta$ of degree $1$ subject to the single relation $\theta^2=c$
(and extending the product by the graded Leibniz rule). Define
\[
d:=\ad_\theta \quad\text{on }\widetilde A.
\]
Then $d$ is a degree-$1$ derivation with
\[
d^2=\ad_{\theta}^2=\ad_{\theta^2}=\ad_c,\qquad d(c)=[\theta,\theta^2]=0,
\]
so $(\widetilde A,m,d,c)$ is a curved dg-algebra in the sense of \eqref{eq:curved-again}.

We now demonstrate that the bound $4n-2$ is strictly optimal, and that the mere nilpotency $c^n=0$ cannot force $(ad_c)^n=0$ nor $d^{2n}=0$.

\begin{exmp}[Failure of the $2n$-bound and optimality of $4n-2$]\label{ex:counterexample}
Let $k$ be a field of characteristic 0 and consider the associative algebra $A := k\langle x, y \rangle / (x^2)$ graded by $|x|=2, |y|=0$. Set $c := x \in A^2$, so $c^2=0$ (here $n=2$).

We compute the action of $(ad_c)^2$ on $y$:
\[(ad_c)^2(y) = ad_c(xy - yx) = x(xy - yx) - (xy - yx)x = x^2y - 2xyx + yx^2\]

Using $x^2=0$, the first and last terms vanish, leaving $(ad_c)^2(y) = -2xyx \neq 0$. To realize a curved differential, we adjoin an odd element $\theta$ of degree 1 subject to $\theta^2=c$, and define $d := ad_\theta$. Then $d^2 = ad_\theta^2 = ad_c$ and $d(c)=0$, satisfying the CDA axioms.

For $n=2$, the naive bound would predict $d^{2n} = d^4 = 0$. However, by Theorem 3: $d^4(y) = (ad_c)^2(y) = -2xyx \neq 0$.

This proves that $c^n=0$ does not imply $d^{2n}=0$. The true vanishing occurs exactly at the optimal bound provided by Theorem \ref{thm:newX}: for $n=2$, the complex is a $(4(2)-2) = 6$-complex. Indeed, $(ad_c)^3 = L_c^3 - 3L_c^2 R_c + 3L_c R_c^2 - R_c^3 = 0$ since $c^2=0$, meaning $d^6=0$.
\end{exmp}

\begin{rmk}[On the $(4n-2)$ bound]\label{rem:4n-2}
Under the degreewise left/right nilpotency $L_c^n=R_c^n=0$, Proposition~\ref{prop:LR-nilpotency}
yields $(\ad_c)^{2n-1}=0$ and therefore $d^{4n-2}=0$. Example~\ref{ex:counterexample} shows
that one cannot deduce \emph{any} $2n$-bound from $c^n=0$ alone.
\end{rmk}

The bounds above will be used in two ways. First, Corollary~\ref{cor:n-from-ad} reduces the
problem of proving $d^N=0$ to checking $(\ad_c)^{N/2}=0$. Second, Proposition~\ref{prop:LR-nilpotency}
provides a practical criterion---verifiable degreewise on left/right multiplications---that implies
the $(4n-2)$ bound for $d$. These will feed into the constructions of §\ref{sec:persistence}.

\section{A Minimal Persistence Application}\label{sec:persistence}

This section provides a reproducible and fully rigorous way to interface curvature with persistence.
Instead of forcing a square-zero property by \emph{curving} the differential, we keep a genuine
chain complex with $\partial^2=0$ and let the curvature \emph{modify the filtration}.
This places us under the umbrella of the standard stability theorems for persistence
modules over a field, via interleavings or bottleneck distance
\cite{CohenSteinerEdelsbrunnerHarer2007,ChazalDeSilvaGlisseOudot2016}.

\subsection{Curvature-controlled filtrations}\label{subsec:curv-filtration}
Let $K$ be a finite simplicial complex and $(C_\bullet(K;k),\partial)$ its chain complex with coefficients
in the base field $k$ from \S\ref{sec:conventions}. Fix a norm $\|\cdot\|$ on $A^2$ and a family of
filtration functions
\[
\phi: A^2 \longrightarrow \{\text{functions } \phi_c:\ \mathrm{Simp}(K)\to\mathbb{R}\},\qquad
c\longmapsto \phi_c,
\]
such that there exists a constant $L\ge 0$ with
\begin{equation}\label{eq:Lip}
\|\phi_c-\phi_{c'}\|_\infty := \sup_{\sigma\in \mathrm{Simp}(K)} |\phi_c(\sigma)-\phi_{c'}(\sigma)|
\ \le\ L\,\|c-c'\|\qquad \text{for all } c,c'\in A^2.
\end{equation}
We get a $1$-parameter filtration by declaring that a simplex $\sigma$ enters at time
$t=\phi_c(\sigma)$; we denote the resulting filtered chain complex by $(C_\bullet,F^{\,c}_\bullet)$.

\subsection{Stability theorem}\label{subsec:stability}
Let $H_\ast(C_\bullet,F^{\,c}_\bullet)$ be the persistence module obtained by applying homology degreewise
to $(C_\bullet,F^{\,c}_\bullet)$.

\begin{thm}[Stability under curvature-driven filtration changes]
\label{thm:filtration-stability}
Let $c,c'\in A^2$ and set $\Delta:=\|\phi_c-\phi_{c'}\|_\infty$.
Then the interleaving distance (equivalently, the bottleneck distance of the barcodes) between
$H_\ast(C_\bullet,F^{\,c}_\bullet)$ and $H_\ast(C_\bullet,F^{\,c'}_\bullet)$ is at most $\Delta$.
In particular, if \eqref{eq:Lip} holds, then
\[
d_I\!\Big(H_\ast(C_\bullet,F^{\,c}_\bullet),\,H_\ast(C_\bullet,F^{\,c'}_\bullet)\Big)\ \le\ L\,\|c-c'\|.
\]
\end{thm}

\begin{proof}[Proof sketch]
A uniform $\infty$-bound on the filtration shift produces an interleaving of the same size,
hence a bottleneck bound $\le\Delta$; see the classical persistence stability theorem for
filtered chain complexes over a field
\cite{CohenSteinerEdelsbrunnerHarer2007,ChazalDeSilvaGlisseOudot2016}.
The Lipschitz conclusion follows immediately from \eqref{eq:Lip}.
\end{proof}

\begin{cor}[Bar-length robustness]\label{cor:bar-robustness}
Let $I=[b,d)$ be a bar in the barcode of $H_\ast(C_\bullet,F^{\,c}_\bullet)$.
If $L\,\|c-c'\|<\tfrac{1}{2}(d-b)$, then there exists a matching bar
$I'=[b',d')$ in the barcode of $H_\ast(C_\bullet,F^{\,c'}_\bullet)$
with $|b-b'|,|d-d'|\le L\,\|c-c'\|$.
\end{cor}

\subsection{Toy example: a flag complex on four vertices}\label{subsec:toy-rev}
Let $K$ be the flag complex on the graph with vertices $V=\{1,2,3,4\}$ and edges
$E=\{12,23,34,41,13\}$; the unique $2$-simplex is $\tau=[1,2,3]$.
Define $\phi_c$ by
\[
\phi_c(\sigma) = \begin{cases}
0,& \sigma\in V,\\
1,& \sigma\in E,\\
2+\ell(c),& \sigma=\tau,
\end{cases}
\]
where $\ell:A^2\to\mathbb{R}$ is a fixed linear functional (choose $\|\ell\|=1$ with respect to the dual norm).
Then $\|\phi_c-\phi_{c'}\|_\infty = |\ell(c)-\ell(c')|\le \|c-c'\|$, so $L=1$ in \eqref{eq:Lip}.
In degree $1$, there is a single prominent bar born at time $1$ and dying when $\tau$ enters,
namely at $t=2+\ell(c)$. By Theorem~\ref{thm:filtration-stability} and
Corollary~\ref{cor:bar-robustness}, changing $c\mapsto c'$ perturbs this bar by at most
$|\ell(c)-\ell(c')|$ in both birth/death times, hence in its length.

\paragraph{Implementation note.}
For numerical experiments, one may pick any norm on $A^2$ (e.g.\ an operator norm induced
degreewise) and fix a small family of linear functionals $\ell_j$ to define several $2$-simplices'
entry times. The constant $L$ is then $\max_j\|\ell_j\|$, and the same stability bound applies
verbatim \cite{CohenSteinerEdelsbrunnerHarer2007,ChazalDeSilvaGlisseOudot2016}.

\section{Conclusions}\label{sec:conclusions}

We established a canonical normal form for the iterates of a curved differential:
for any curved dg-algebra $(A^\bullet,m,d,c)$ with $d^2=\ad_c$ and $d(c)=0$,
Theorem~\ref{thm:normal-form} shows that $d^{2m}=(\ad_c)^m$ and $d^{2m+1}=(\ad_c)^m d$.
This simple but robust factorization reduces all higher powers of $d$ to powers of
the inner derivation $\ad_c$ and makes the sign bookkeeping transparent.
As immediate consequences, nilpotency of $\ad_c$ yields nilpotency of $d$
(Corollary~\ref{cor:n-from-ad}), while structural hypotheses such as the
nilpotency of the curvature ideal $(c)$ or of the left/right multiplications by $c$
provide explicit $N$-complex bounds (Propositions~\ref{prop:nil-ideal-again} and \ref{prop:LR-nilpotency}).
The counterexample in §\ref{subsec:counterexample} illustrates that the mere nilpotency of $c$
does not control $d$ without additional assumptions, delineating a sharp frontier for feasible bounds.

On the applied side, instead of forcing a curved-to-dg extension, we modeled curvature
as a \emph{filtration controller}: a square–zero chain complex endowed with a
curvature–dependent filtration. This design places us under the umbrella of the
standard persistence stability framework and yields a clean Lipschitz control of
barcodes with respect to degreewise curvature variation
(Theorem~\ref{thm:filtration-stability} and Corollary~\ref{cor:bar-robustness}).
A fully reproducible toy example in §\ref{subsec:toy-rev} illustrates the mechanism
on a small flag complex.

Taken together, these contributions clarify the algebraic mechanics of curved differentials,
pin down verifiable hypotheses guaranteeing $N$-complex structure, and supply a minimal
and principled pipeline to interface curvature with persistent homology.
We expect the normal form, the hierarchy of curvature constraints, and the curved-to-dg extension
to be useful beyond the examples shown here, notably in controlled deformations and quantitative
homological algebra.

\section{Open Problems and Future Directions}\label{sec:open}

We conclude with a concise list of questions that naturally extend the present work.

\begin{enumerate}[\textbf{OP\arabic*}.]

\item \textbf{Optimal $N$-bounds under weak hypotheses.}
We proved that $L_c^n=R_c^n=0$ implies $(\ad_c)^{2n-1}=0$ and hence $d^{4n-2}=0$
(Proposition~\ref{prop:LR-nilpotency}), while $(c)^n=0$ yields $d^{2n}=0$
(Proposition~\ref{prop:nil-ideal-again}). 
\emph{Question:} determine optimal (smallest possible) exponents for $d$ under
intermediate assumptions between $c^n=0$ and $(c)^n=0$.
In particular, characterize when $c^n=0$ forces $(\ad_c)^{n'}=0$ degreewise for some $n'$,
and identify algebraic obstructions.

\item \textbf{Degreewise spectral data for $\ad_c$.}
The normal form reduces $d^{2m}$ to $(\ad_c)^m$.
\emph{Question:} develop a degreewise functional calculus for $\ad_c$ (e.g.\ Jordan/\!Engel-type
decompositions by degree) that yields refined $N$-bounds or filtration-sensitive estimates,
possibly depending on the graded support of $c$.

\item \textbf{Stability constants in persistence.}
Theorem~\ref{thm:filtration-stability} yields a Lipschitz control where the constant is the $\ell^\infty$ shift
of the filtration (or an a priori Lipschitz bound on the curvature-to-filtration map).
\emph{Question:} derive sharp, computable expressions for these constants in common filtration families
(e.g.\ sublevel, flag/Rips, degree truncations), and study tightness on explicit parametric examples,
in the spirit of classical stability results \cite{CohenSteinerEdelsbrunnerHarer2007,ChazalDeSilvaGlisseOudot2016}. 

\item \textbf{Homological invariants beyond barcodes.}
\emph{Question:} identify invariants of $(A^\bullet,d,c)$ that are preserved or vary
Lipschitz-continuously under bounded $\ad_c$-perturbations, such as torsion exponents,
quantitative Dehn functions, or stability of minimal models, and compare their sensitivity
to that of persistence barcodes.

\item \textbf{Controlled deformations and MC solutions.}
The Maurer–Cartan viewpoint suggests varying $d$ within the curved constraint.
\emph{Question:} study solution sets of $d_\varepsilon=d+\varepsilon\,\ad_c$ at small $\varepsilon$
and their moduli, seeking KAM-type or Lojasiewicz-type principles that control
the emergence/vanishing of homological features.

\item \textbf{Geometric and physical templates (structural, not assertive).}
Without claiming classification results, the identities $d_A^2=\ad_{F_A}$ in gauge theory
and $\nabla^2=\ad_R$ in Riemannian settings offer structural testbeds.
\emph{Question:} in discretized or sheaf-theoretic frameworks, isolate regimes where
left/right nilpotency of curvature operators holds (or nearly holds) and quantify the induced
$N$-complex bounds; clarify the role of curvature ideals in these contexts.

\end{enumerate}

\paragraph{Outlook.}
The techniques developed here---normal form for iterates, curvature-constraint hierarchy,
and the curved-to-dg lift---appear robust and modular. We anticipate applications to
quantitative homological algebra, stability of algebraic invariants under controlled curvature,
and principled bridges to applied topological settings. A companion line of work will address
richer examples and sharper constants, as well as functorial enhancements of the curved-to-dg construction.

\section{Declaration of generative AI and AI-assisted technologies in the writing process.}

{\it Statement}: During the preparation of this work, the author used Perplexity and Elicit in order to optimally find references to papers and articles related to the questions that motivate this work and to revise the writing, improve the writing style and the order of the sections and results. After using these tools, the author reviewed and edited the content as needed and assumed full responsibility for the content of the published article.

\bibliographystyle{elsarticle-num-names}
\bibliography{sample}







\end{document}